\documentclass[a4paper]{amsart}
\usepackage{amscd}
\usepackage{amsthm}
\usepackage{graphicx}
\usepackage{amsmath}
\usepackage{amsfonts}
\usepackage{amssymb}%
\providecommand{\U}[1]{\protect\rule{.1in}{.1in}}
\newtheorem{theorem}{Theorem}

\newtheorem{corollary}[theorem]{Corollary}

\newtheorem{definition}[theorem]{Definition}

\newtheorem{remark}[theorem]{Remark}

\begin{document}

\title[A multiplicative measure on the positive real axis]{A multiplicative measure on \\ the positive real axis}
\author{Pablo Rocha}

\begin{abstract} We construct a measure $\mu$ and a $\sigma$-algebra $\mathcal{M}$ of subsets of the positive real axis,
$\mathbb{R}_{>0}$, with the following multiplicative property:
\[
\mu \left( \bigcup_j E_j \right) = \prod_j \mu(E_j)
\]
for every countable collection $\{ E_j \}$ of pairwise disjoint subsets of $\mathcal{M}$. For them, we apply the 
Carath\'eodory's procedure to the triplet $\left( \mathbb{R}_{>0}, \cdot \, , \tau \right)$, where $\cdot$ is the product of 
$\mathbb{R}$ and $\tau$ is the usual topology on $\mathbb{R}_{>0}$.

We conclude this note describing the connection between this multiplicative measure $\mu$ and the Lebesgue measure.

\end{abstract}
\maketitle

\section{Introduction}

Roughly speaking, the Carath\'eodory's procedure in constructing the Lebesgue measure in $\mathbb{R}$ is the following: first, it is built a outer measure on $\mathbb{R}$ from the length of the closed intervals of $\mathbb{R}$. Once the outer measure is constructed, the next step is to define the Lebesgue measurable sets. The collection of measurable subsets of $\mathbb{R}$ turns out to be a $\sigma$-algebra. Then the Lebesgue measure is defined as the outer measure restricted on the $\sigma$-algebra of Lebesgue measurable subsets of $\mathbb{R}$. Such a measure, say $m$, satisfies 
\[
m \left( \bigcup_j E_j \right) = \sum_j m(E_j)
\]
for every countable collection $\{ E_j \}$ of pairwise disjoint measurable sets. This pro-perty of $m$ is called countable additivity 
(see, e.g. \cite{zygmund}).

The purpose of this note is to construct a \textit{multiplicative measure} on the positive real axis $\mathbb{R}_{>0} = (0, +\infty)$. For them, we adapt the Carath\'eodory's procedure to our new setting. We consider the positive real axis $\mathbb{R}_{>0}$ with the usual topology and define the "length" of each closed interval $I=[a, b]$ in $\mathbb{R}_{>0}$ by $$\ell(I) = b \cdot a^{-1}.$$ Then, to apply the 
Carath\'eodory's idea with these elements, we will obtain a measure $\mu$ and a $\sigma$-algebra $\mathcal{M}$ in $\mathbb{R}_{>0}$ with the following properties:
\[
\mu(\emptyset) = 1, \,\,\,\,\,\,\, \mu((0,1)) = +\infty, \,\,\,\,\,\,\, \mu(\mathbb{R}_{>0}) = +\infty,
\]
\[
\mu \left( \bigcup_j E_j \right) = \prod_j \mu(E_j)
\]
for every countable collection $\{ E_j \}$ of pairwise disjoint sets of $\mathcal{M}$. We call to this property: 
\textit{countable multiplicativity}. By this reason, we say that $\mu$ is a \textit{multiplicative measure}.

Once constructed the measure $\mu$ and the $\sigma$-algebra $\mathcal{M}$, we will show that $\mathcal{M}$ coincides with the $\sigma$-algebra of the Lebesgue measurable subsets of $\mathbb{R}_{>0}$ and $\mu(E) = \lambda(E)$ for each $E \in \mathcal{M}$,
where
\[
\lambda(E) = \exp\left(\int_{E} \frac{1}{x} \, dx \right). 
\]

\

In Section 2 we give some preliminaries about single and double infinite products of positive real numbers which are necessary to our construction. The multiplicative measure $\mu$ is constructed in Section 3. Finally, in Section 4, we show that $\mu(E) = \lambda(E)$ for each $E \in \mathcal{M}$.

\

{\bf Notation.} The positive real axis will be denoted by $\mathbb{R}_{>0} = (0, +\infty)$ and  $\mathbb{R}_{\geq 1} = [1, +\infty)$.

\section{Preliminaries}

In this section we present some basic facts on certain single and double infinite products of real numbers (see, e.g. \cite{apostol}).

\subsection{Single infinite products} Since we will consider only infinite products with positive factors, the following
simplified definition will be adopted.

\

Given a sequence $\{ a_j \}$ of positive real numbers, let
\[
p_N = \prod_{j=1}^{N} a_j = a_1 \cdot a_2 \cdot \cdot \cdot a_N.
\]
The number $p_N$ is called the \textit{$N$th partial product} of the sequence $\{ a_j \}$.

\begin{definition} We say that an infinite product of positive real numbers $\displaystyle{\prod_{j=1}^{+\infty}} a_j$ converges if there exists a number $p >0$ such that

\[
\lim_{N \to +\infty} p_N = p.
\]
In this case, $p$ is called the value of the infinite product and we write $p = \displaystyle{\prod_{j=1}^{+\infty}} a_j$.
\end{definition}

\begin{remark} If $p = 0$ or $p = +\infty$ or the limit does not exist we say that the product diverges. A necessary
condition for convergence of an infinite product is $\displaystyle{\lim_{j \to +\infty}} a_j =1$, but this is not sufficient.
\end{remark}

\begin{remark}\label{suc mayor a 1} If $\{ a_j \}_{j=1}^{+\infty}$ is a sequence of real numbers greater than or equal to $1$, then
$1 \leq p_N \leq p_{N+1}$ for all $N \geq 1$. So, $\prod_{j=1}^{+\infty} a_j = \sup \{ p_N : N \in \mathbb{N}\} \in [1, +\infty]$. This is,
$\prod_{j=1}^{+\infty} a_j$ converges in $\mathbb{R}_{\geq 1}$ or $\prod_{j=1}^{+\infty} a_j = +\infty$.
\end{remark}

The following theorem on rearrangements is an immediate consequence of Theorem 3.3 in \cite{moorthy} and from remark \ref{suc mayor a 1}.

\begin{theorem}\label{theo rearr} If $\{ a_j \}_{j=1}^{+\infty}$ is a sequence of real numbers greater than or equal to $1$, then 
\[
\prod_{k=1}^{+\infty} a_{\sigma(k)} = \prod_{j=1}^{+\infty} a_{j}
\]
for each bijection $\sigma : \mathbb{N} \to \mathbb{N}$.
\end{theorem}

\

The proof of Theorem 4 is simpler. Indeed, for $a_j \geq 1$, we have that
\[
\prod_j a_j = \exp\left( \sum_j \log(a_j) \right) = \exp\left( \sum_k \log(a_{\sigma(k)}) \right) = \prod_j a_{\sigma(k)}. 
\] 

\subsection{Double infinite products}

Formally, a double infinite product of positive real numbers is of the form
\[
\prod_{i,j=1}^{+\infty} a_{ij}
\]
where the $a_{i j}$'s are indexed by  pair of natural numbers $i, j \in \mathbb{N}$.

\

A sequence of real numbers of the form $\{ a_{ij} \}$ is called a double sequence.

\

Our main interest is in giving a definition of a double infinite product of real numbers greater than or equal to $1$ that does not depend on the order of its factors. 

\

If $F \subset \mathbb{N} \times \mathbb{N}$ is a finite subset of pairs of natural numbers, then we denote by
\[
\prod_{F} a_{ij} = \prod_{(i,j) \in F} a_{ij}
\]
the partial product of all $a_{ij}$'s whose indices $(i,j)$ belong to $F$.

\begin{definition} The unordered infinite product of real numbers $a_{ij} \geq 1$ is
\[
\prod_{\mathbb{N} \times \mathbb{N}} a_{ij} = \sup_{F \in \mathcal{F}} \left\{  \prod_{F} a_{ij}  \right\}
\]
where the supremum is taken over the collection $\mathcal{F}$ of all finite subsets $F \subset \mathbb{N} \times \mathbb{N}$.
\end{definition}

Such an unordered infinite product converges if the supremum is finite and diverges to $+\infty$ if the supremum is $+\infty$.
Note that this supremum exists in $\mathbb{R}_{\geq 1}$ if and only if the finite partial products are bounded from above.

\

Next, we define rearrangements of a double infinite product into single infinite product and show that every rearrangement of a 
convergent unordered double infinite product converges to the same product.

\

\begin{definition} A rearrangement of a double infinite product of positive real numbers 
\[
\prod_{i,j}^{+\infty} a_{ij}
\]
is a single infinite product of the form
\[
\prod_{k=1}^{+\infty} b_k, \,\,\,\,\,\, b_k = a_{\sigma{(k)}}
\]
where $\sigma : \mathbb{N} \to \mathbb{N} \times \mathbb{N}$ is a one-to-one and onto map.
\end{definition}

The following result is an adaptation of Theorem 8.42 in \cite{apostol}, pp. 201, to the context of double products.

\begin{theorem}\label{theo rearr 2} If the unordered product of a double sequence of real numbers greater than or equal to $1$ converges,
then every rearrangement of the double infinite product into a single infinite product converges to the unordered product.
\end{theorem}

\begin{proof} Suppose that the unordered product $\prod a_{ij}$, $a_{ij} \geq 1$, converges with
\[
\prod_{\mathbb{N} \times \mathbb{N}} a_{ij} = s \in \mathbb{R}_{\geq 1},
\]
and let
\[
\prod_{k=1}^{+\infty} b_k
\]
be a rearrangement of the double product corresponding to a map $\sigma : \mathbb{N} \to \mathbb{N} \times \mathbb{N}$.
For $N \in \mathbb{N}$, let
\[
F_N = \{ \sigma(k) \in \mathbb{N} \times \mathbb{N} : 1 \leq k \leq N \},
\]
so that
\[
\prod_{k=1}^{N} b_k = \prod_{F_N} a_{ij}.
\]

Given $\epsilon > 0$, let $F \subset \mathbb{N} \times \mathbb{N}$ be a finite set such that
\[
s - \epsilon < \prod_{F} a_{ij} \leq s,
\]
and let $K \in \mathbb{N}$ be defined by $K = \max \{ \sigma^{-1}(i,j) : (i,j) \in F \}$. 

If $N \geq K$, then $F_{N} \supseteq F$ and since $a_{ij} \geq 1$, we obtain
\[
s - \epsilon < \prod_{F} a_{ij} \leq \prod_{F_N} a_{ij} \leq s.
\]
This implies that
\[
\left| \prod_{k=1}^{N} b_k - s \right| \leq \epsilon, \,\,\,\,\,\,\,\, \forall \, N \geq K.
\]
Thus,
\[
\prod_{k=1}^{+\infty} b_k = \prod_{\mathbb{N} \times \mathbb{N}} a_{ij}.
\]
\end{proof}

The rearrangement of a double infinite product into a single infinite product is one natural way to
interpret a double infinite product in terms of single infinite products. Another way is to use iterated
products of single infinite product.

Given a double product $\prod a_{ij}$, one can define two iterated products, obtained by multiplying first over one index followed by the other:
\[
\prod_{i=1}^{+\infty} \left( \prod_{j=1}^{+\infty} a_{ij} \right) = \lim_{N \to +\infty} \prod_{i=1}^{N} \left( 
\lim_{M \to +\infty} \prod_{j=1}^{M} a_{ij} \right),
\]
\[
\prod_{j=1}^{+\infty} \left( \prod_{i=1}^{+\infty} a_{ij} \right) = \lim_{M \to +\infty} \prod_{j=1}^{M} \left( 
\lim_{N \to +\infty} \prod_{i=1}^{M} a_{ij} \right).
\]

\begin{theorem}\label{iter} A unordered product of real numbers $a_{ij} \geq 1$ converges if and only if either one of the iterated products
\[
\prod_{i=1}^{+\infty} \left( \prod_{j=1}^{+\infty} a_{ij} \right), \,\,\,\,\,\, \prod_{j=1}^{+\infty} \left( \prod_{i=1}^{+\infty} a_{ij} \right)
\]
converges. Moreover, both iterated products converge to the unordered product:
\[
\prod_{\mathbb{N} \times \mathbb{N}} a_{ij} = \prod_{i=1}^{+\infty} \left( \prod_{j=1}^{+\infty} a_{ij} \right) =
\prod_{j=1}^{+\infty} \left( \prod_{i=1}^{+\infty} a_{ij} \right)
\]
\end{theorem}

\begin{proof} Suppose that the unordered product converges. Since $a_{ij} \geq 1$ and $\sup_{I} \left( \sup_{J} c_{ij} \right) =
\sup_{I \times J} c_{ij}$ for each double sequence of reals, we obtain
\[
\prod_{i=1}^{+\infty} \left( \prod_{j=1}^{+\infty} a_{ij} \right) = \sup_{N} \left\{\prod_{i=1}^{N} \left( \sup_{M}\prod_{j=1}^{M} a_{ij} \right) \right\}
= \sup_{(N, M) \in \mathbb{N} \times \mathbb{N}} \left\{ \prod_{i=1}^{N} \prod_{j=1}^{M} a_{ij} \right\} \leq 
\prod_{\mathbb{N} \times \mathbb{N}} a_{ij}.
\]

Conversely, suppose that one of the iterated products exists. Without loss of generality, we suppose that
\[
\prod_{i=1}^{+\infty} \left( \prod_{j=1}^{+\infty} a_{ij} \right) < +\infty.
\]
Let $F \subset \mathbb{N} \times \mathbb{N}$ be a finite subset. Then, there exist two natural numbers $N$ and $M$ such that
\[
F \subset \{1, 2, ..., N \} \times \{ 1, 2, ..., M \} =: R.
\]
So that,
\[
\prod_{F} a_{ij} \leq \prod_{R} a_{ij} = \prod_{i=1}^{N} \left( \prod_{j=1}^{M} a_{ij} \right) \leq 
\prod_{i=1}^{+\infty} \left( \prod_{j=1}^{+\infty} a_{ij} \right).
\]
Therefore, the unordered product converges and
\[
\prod_{\mathbb{N} \times \mathbb{N}} a_{ij} \leq \prod_{i=1}^{+\infty} \left( \prod_{j=1}^{+\infty} a_{ij} \right).
\]
\end{proof}

%{\bf Remark.} It is obtained the same result if we define the unordered sum for a series of non-negative terms. Indeed, for $a_{ij} \geq 1$, we have that
%\[
%\prod_{\mathbb{N} \times \mathbb{N}} a_{ij} = \exp \left(  \sum_{\mathbb{N} \times \mathbb{N}} \log(a_{ij}) \right).
%\]
%So, Theorem 7 and theorem 8 follow.

\section{Main Results}

To construct our multiplicative measure $\mu$, we consider the positive real axis $\mathbb{R}_{>0}$ with the usual topology and for each closed interval $I=[a,b]$ in $\mathbb{R}_{>0}$, we define its "length", $\ell(I)$, by
\[
\ell(I) = b \cdot a^{-1},
\]
where $\cdot$ is the product of $\mathbb{R}$. It is easy to check that

\

1. $\ell(I) \geq 1$ for each closed interval $I = [a,b]$ in $\mathbb{R}_{>0}$.

\

2. If $I$, $J_1$ and $J_2$ are three closed intervals in $\mathbb{R}_{>0}$ such that $I \subseteq J_1 \bigcup J_2$, then 
$\ell(I) \leq \ell(J_1) \cdot \ell(J_2)$.

\

In the sequel, a cover of a set $E \subseteq \mathbb{R}_{>0}$ is a countable collection $S$ of intervals $I =[a,b]$, with 
$0 < a < b < +\infty$, such that $E \subset \bigcup_{I \in S} I$.

\

Now we shall construct the exterior measure of an arbitrary subset $E$ of $\mathbb{R}_{>0}$. Given a set $E \subseteq \mathbb{R}_{>0}$, 
we cover $E$ by a countable collection of intervals $S= \{ I_j \}$ in $\mathbb{R}_{>0}$, and let
\[
\nu(S) = \prod_{I_j \in S} \ell(I_j).
\]
We point out that the value $\nu(S)$ is independent of the rearrangement of the cover $S=\{ I_j \}$ of $E$ (see Theorem \ref{theo rearr}). The exterior measure of $E$, denoted $\mu_{e}(E)$, is defined by
\[
\mu_e (E) = \inf_{S} \nu(S),
\]
where the infimum is taken over all such covers $S$ of $E$. Thus, $1 \leq \mu_e (E) \leq + \infty$. Moreover, if $E_1 \subset E_2 \subset 
\mathbb{R}_{>0}$, then $\mu_e (E_1) \leq \mu_e(E_2)$.

\begin{theorem}\label{intervalo} Let $I =[a,b] \subset\mathbb{R}_{>0}$. Then $\mu_e (I) = \ell(I)$. 
\end{theorem}

\begin{proof} That $\mu_e(I) \leq \ell(I)$ is obvious. To show the opposite inequality we choose a countable cover $S=\{ I_j \}$ of $I$.
Let $0 < \epsilon < 1$ be fixed. For each $j \in \mathbb{N}$, let $I_{j}^{\ast}$ be an interval of $\mathbb{R}_{>0}$ such that its interior 
$(I_{j}^{\ast})^{\circ} \supset I_j$ and $\ell(I_{j}^{\ast}) \leq \frac{1 + \epsilon^{j}}{1 + \epsilon^{j+1}} \cdot \ell(I_j)$. Then
$I \subset \bigcup_j (I_{j}^{\ast})^{\circ}$. Since $I$ is closed and bounded, by the Heine-Borel Theorem, there exists $N \in \mathbb{N}$ such that $I \subset \bigcup_{j=1}^{N} I_{j}^{\ast}$. Now, the property 2. above implies that 
$\ell(I) \leq \prod_{j=1}^{N} \ell(I_{j}^{\ast})$. Therefore,
\[
\ell(I) \leq \frac{1 + \epsilon}{1 + \epsilon^{N+1}} \prod_{j=1}^{N} \ell(I_j) \leq (1 + \epsilon) \cdot \nu(S).
\]
Then letting $\epsilon \to 0^{+}$, we obtain $\ell(I) \leq \nu(S)$ and hence $\ell(I) \leq \mu_e(I)$.
\end{proof}

\begin{theorem}\label{sub multi} If $E = \bigcup_j E_j$ is a countable union of subsets of $\mathbb{R}_{>0}$, then $$\mu_e (E) \leq \prod_j \mu_e(E_j).$$
\end{theorem}

\begin{proof} Without loss of generality, we may assume that $\prod_j \mu_e(E_j) < +\infty$. Now, we fix $0 < \epsilon < 1$. 
Given $j \in \mathbb{N}$, we choose intervals $I_{ij}$ of $\mathbb{R}_{>0}$ such that $E_j \subset \bigcup_{i} I_{ij}$ and 
\begin{equation}\label{ineq}
\prod_{i} \ell(I_{ij}) < \frac{1+\epsilon^{j}}{1+ \epsilon^{j+1}}\mu_{e}(E_j).
\end{equation}
Let $\sigma : \mathbb{N} \to \mathbb{N} \times \mathbb{N}$ be a bijection, then the collection $\{ J_k \}$ defined by $J_k = I_{\sigma(k)}$ is a cover of $E$. So that,
\begin{equation}\label{submulti 2} 
\mu_{e}(E) \leq \prod_{k} \ell(J_k) = \prod_{\mathbb{N} \times \mathbb{N}} \ell(I_{ij}) = \prod_j \left( \prod_i \ell(I_{ij}) \right)
\leq (1+ \epsilon) \prod_{j} \mu_e(E_j),
\end{equation}
where the first equality follows from Theorem \ref{theo rearr 2}, the second from Theorem \ref{iter} and the last inequality from 
(\ref{ineq}) and 
\[
\prod_{j=1}^{N} \frac{1+\epsilon^{j}}{1+ \epsilon^{j+1}}\mu_{e}(E_j) = \frac{1+\epsilon}{1+ \epsilon^{N+1}} \prod_{j=1}^{N} \mu_e(E_j) \to
(1+ \epsilon) \prod_{j=1}^{+\infty} \mu_e(E_j).
\]
Finally, the result follows by letting $\epsilon \to 0$ in (\ref{submulti 2}).
\end{proof}

The next theorem relates the exterior measure of an arbitrary set of $\mathbb{R}_{>0}$ with the exterior measure of open sets and 
$G_{\delta}$ sets of $\mathbb{R}_{>0}$. This theorem is an analogous of the theorems 3.6 and 3.8 in \cite{zygmund}, pp. 44.
 
\begin{theorem}\label{G delta} Let $E \subseteq \mathbb{R}_{>0}$, then
\begin{enumerate}
\item[(i)] given $0 < \epsilon < 1$, there exists and open set $G$ such that $E \subset G$ and $\mu_e(G) \leq (1+ \epsilon) \cdot \mu_e(E)$;

\item[(ii)] there exists a set $H$ of type $G_{\delta}$ in $\mathbb{R}_{>0}$ such that $E \subset H$ and $\mu_e(E) = \mu_e(H)$.
\end{enumerate}
\end{theorem}

\begin{proof} It is easy to check that $(i)$ implies $(ii)$. To prove $(i)$, we fix $0 < \epsilon < 1$ and choose a cover $\{ I_j \}$ of $E$ in $\mathbb{R}_{>0}$ such that $\prod_j \ell(I_j) \leq \sqrt{(1 + \epsilon)} \cdot \mu_e(E)$. For each $j \in \mathbb{N}$, 
let $I_{j}^{\ast}$ be an interval of $\mathbb{R}_{>0}$ such that its interior $(I_{j}^{\ast})^{\circ} \supset I_j$ and
$\ell(I_{j}^{\ast}) \leq \frac{\sqrt{(1 + \epsilon^{j})}}{\sqrt{(1 + \epsilon^{j+1})}} \cdot \ell(I_j)$. If 
$G = \bigcup_j (I_{j}^{\ast})^{\circ}$, then $G$ is an open set such that $E \subset G$ and
\[
\mu_e(G) \leq \prod_j \ell(I_{j}^{\ast}) \leq \sqrt{(1 + \epsilon)} \cdot \prod_j \ell(I_j) \leq (1+\epsilon) \cdot \mu_{e}(E).
\]
\end{proof}

The following corollary is an immediate consequence of Theorem \ref{G delta}, $(i)$.

\begin{corollary}
Let $E \subset \mathbb{R}_{>0}$. Then $$\mu_e(E) = \inf \mu_e(G),$$ where the infimum is taken over all open sets $G$ of $\mathbb{R}_{>0}$ containing $E$.
\end{corollary}

A set $E \subseteq \mathbb{R}_{>0}$ is said to be \textit{measurable} if given $\epsilon > 0$, there exists an open set $G$ of 
$\mathbb{R}_{>0}$ such that $E \subset G$ and
\[
\mu_e (G \setminus E) < 1 + \epsilon.
\]
If $E \subseteq \mathbb{R}_{>0}$ is measurable, its exterior measure is called its measure, which we denote by $\mu(E)$, i.e.:
\[
\mu(E) = \mu_e(E), \,\,\, \text{for measurable} \,\, E.
\]
We say that a set $E$ is $\mu$-measurable if satisfies the above definition.

\

The next list of examples and properties of measurable sets of $\mathbb{R}_{>0}$ follows from the previous results and the definition of measurability.

\

$1)$ Every open set of $\mathbb{R}_{>0}$ is $\mu$-measurable. For instance, the empty set and the sets $\mathbb{R}_{>0}$ and $(0,1)$ are $\mu$-measurable with $\mu(\emptyset) = 1$, $\mu(\mathbb{R}_{>0}) = +\infty$ and $\mu((0,1)) = +\infty$.

\

$2)$ Every set of exterior measure $1$ is $\mu$-measurable. For instance, every point $\{ a \}$ of $\mathbb{R}_{>0}$ is $\mu$-measurable with
$\mu(\{ a \}) = 1$.

\

The claim $2)$ follows from Theorem \ref{G delta}, $(i)$. 

\

$3)$ If $E = \bigcup_j E_j$ is a countable union of $\mu$-measurable sets of $\mathbb{R}_{>0}$, then $E$ is $\mu$-measurable and
\[
\mu(E) \leq \prod_j \mu(E_j).
\]
Indeed, given $0 < \epsilon < 1$, for each $j=1, 2, ...$, we take an open set $G_j$ such that $E_j \subset G_j$ and
$\mu_e(G_j \setminus E_j) < \frac{1 + \epsilon^{j}}{1 + \epsilon^{j+1}}$. Then $G = \bigcup_j G_j$ is an open such that $E \subset G$,
and since $G \setminus E \subset \bigcup_j (G_j \setminus E_j)$, we have
\[
\mu_e(G \setminus E) \leq \mu_e \left( \bigcup_j (G_j \setminus E_j) \right) \leq \prod_j \mu_e ((G_j \setminus E_j)) \leq (1+ \epsilon),
\]
where the second inequality follows from Theorem \ref{sub multi}. Thus $E$ is $\mu$-measurable and 
$\mu(E) = \mu_e(E) \leq \prod_j \mu_e(E_j) = \prod_j \mu(E_j)$.

\

$4)$ Every interval $I \subset \mathbb{R}_{>0}$ is $\mu$-measurable. If $I$ is bounded and $\inf I > 0$, then $\mu(I) = \ell(I)$; opposite case 
$\mu(I) = +\infty$.

\

The claim $4)$ follows from Theorem \ref{intervalo}, and from $1)$ and $2)$ above.

\

$5)$ If $\{ I_j \}_{j=1}^{N}$ is a finite collection of nonoverlapping intervals of $\mathbb{R}_{>0}$, then $\bigcup_j I_j$ is $\mu$-measurable and 
$\mu\left( \bigcup_j I_j \right) = \prod_j \mu(I_j)$.

\

The claim $5)$ is an extension of Theorem \ref{intervalo}.

\

$6)$ If $E_1$ and $E_2$ are two subsets of $\mathbb{R}_{>0}$ such that 
\[
d(E_1, E_2) := \inf \{ |x - y| : x \in E_1, y \in E_2 \} > 0,
\]
then $\mu_e(E_1 \bigcup E_2) = \mu_e(E_1) \cdot \mu_e(E_2)$.

\

The proof of the claim $6)$ is similar to that of Lemma 3.16, in \cite{zygmund}, but reformulated to our setting. 

\

$7)$ Every closed set $F$ of $\mathbb{R}_{>0}$ is $\mu$-measurable.

\

To see $7)$, we assume first that $F$ is a compact set in $\mathbb{R}_{>0}$, so $\inf F > 0$ and hence $\mu_e(F) < +\infty$.  Given $0 < \epsilon < 1$, by Theorem \ref{G delta} statement $(i)$, there exists an open set $G$ such that $F \subset G$ and 
$\mu_e(G) \leq (1+ \epsilon) \cdot \mu_e(F)$. Since $G\setminus F$ is open in $\mathbb{R}_{>0}$, Theorem 1.11 in \cite{zygmund}, implies there exists a collection of nonoverlapping closed intervals $\{ I_j \}$, where $I_j =[a_j, b_j]$ with $0 < a_j < b_j$, such that
$G\setminus F = \bigcup_j I_j$. Thus, $\mu_e (G\setminus F) \leq \prod_j \mu_e(I_j)$. Then, to prove the $\mu$-measurability of $F$ it suffices
to show that $\prod_j \mu_e(I_j) \leq 1 + \epsilon$. We have $G = F \cup (\bigcup I_j) \supset F \cup (\bigcup_{j=1}^{N} I_j)$ for each $N$.
Since $F$ and $\bigcup_{j=1}^{N} I_j$ are disjoint and compact, by $6)$, we obtain
\[
\mu_e(G) \geq \mu_e \left( F \cup (\bigcup_{j=1}^{N} I_j) \right) = \mu_e(F) \cdot \mu_e \left( \bigcup_{j=1}^{N} I_j \right).
\]
By $5)$, we have $\prod_{j=1}^{N} \mu_e(I_j) = \mu_e \left( \bigcup_{j=1}^{N} I_j \right) \leq \displaystyle{\frac{\mu_e(G)}{\mu_e(F)}} \leq 1 + \epsilon$. This proves the result in the case when $F$ is compact.

To end, let $F$ be any closed subset of $\mathbb{R}_{>0}$ and write $F = \bigcup_j F_j$, where 
$F_j = F \cap \{ x > 0 : j^{-1} \leq x \leq j \}$, $j=2, 3, ...$. Since, each $F_j$ is compact and, thus, $\mu$-measurable; then the $\mu$-measurability of $F$ follows from $3)$.

\

$8)$ The complement of a $\mu$-measurable set of $\mathbb{R}_{>0}$ is $\mu$-measurable. 

\

The proof of $8)$ is similar to that of Theorem 3.17 in \cite{zygmund}.

\

$9)$ If $\{ E_j \}$ is a countable collection of $\mu$-measurable sets of $\mathbb{R}_{>0}$, then $\bigcap_j E_j$ is $\mu$-measurable.

\

The proof of $9)$ is similar to that of Theorem 3.18 in \cite{zygmund}.

\

$10)$ A set $E$ of $\mathbb{R}_{>0}$ is $\mu$-measurable if and only if given $\epsilon >0$, there exists a closed set $F \subset E$ such that
$\mu_e(E \setminus F) < 1 + \epsilon$.

\

The proof of $10)$ is similar to that of Lemma 3.22 in \cite{zygmund}.

\

The following result follows from $3)$, $8)$ and $9)$.

\begin{theorem} Let $\mathcal{M}$ be the collection of $\mu$-measurable subsets of $\mathbb{R}_{>0}$. Then, $\mathcal{M}$ is a $\sigma$-algebra.
\end{theorem}

From $1)$ and this theorem we obtain the following corollary.

\begin{corollary} Every Borel set of $\mathbb{R}_{>0}$ is $\mu$-measurable.
\end{corollary}

We are now in a position to prove the following theorem.

\begin{theorem} If $\{ E_j \}$ is a countable collection of disjoint $\mu$-measurable subsets of $\mathbb{R}_{>0}$, then
\[
\mu \left( \bigcup_j E_j \right) = \prod_j \mu(E_j).
\]
\end{theorem}

\begin{proof} First, we assume that each $E_j$ is bounded with $\inf E_j > 0$. Given $0 < \epsilon < 1$ and $j=1, 2,...$, by $10)$, there exists a closed $F_j \subset E_j$ such that $\mu(E_j \setminus F_j) < \frac{1+ \epsilon^{j}}{1+ \epsilon^{j+1}}$. Then,
$\mu (E_j) \leq \frac{1+ \epsilon^{j}}{1+ \epsilon^{j+1}} \cdot \mu(F_j)$. Since the $E_j$ are bounded with $\inf E_j >0$ and disjoint, 
the $F_j$ are compact and disjoint. Then, By $6)$, we have $\mu \left( \bigcup_{j=1}^{N} F_j \right) = \prod_{j=1}^{N} \mu(F_j)$, for each
$N$. So,
\[
\mu \left(\bigcup_j E_j \right) \geq \prod_{j=1}^{N} \mu(F_j) \geq \frac{1 + \epsilon^{N+1}}{1 + \epsilon} \prod_{j=1}^{N} \mu(E_j) \geq
\frac{1}{1 + \epsilon} \prod_{j=1}^{N} \mu(E_j),
\]
for each $N$. By letting $N \to + \infty$ and since $0 < \epsilon < 1$ is arbitrary, we obtain $\mu \left(\bigcup_j E_j \right) \geq 
\prod_{j=1}^{+\infty} \mu(E_j)$. The opposite inequality follows from Theorem \ref{sub multi}. This proves the theorem in the case when
the $E_j$ are bounded with $\inf E_j > 0$.

For the general case, let $I_k = [k^{-1}, k]$, $k =1, 2, ...$, and we define $J_1= I_1$ and $J_k = I_k \setminus I_{k-1}$ for $k \geq 2$.
Then the sets $A_{jk} = E_j \cap J_k$, with $j, k=1,2,...$, are bounded, disjoint, $\mu$-measurable and with $\inf A_{jk} >0$. Since
$E_j = \bigcup_{k} A_{jk}$ and $\bigcup_{j} E_j = \bigcup_{j, k} A_{jk}$, by the case already established and proceeding, with the double product, as in the proof of Theorem \ref{sub multi}, we have
\[
\mu \left( \bigcup_j E_j \right) = \mu \left( \bigcup_{j,k}  A_{jk} \right) = \prod_{\mathbb{N} \times \mathbb{N}} \mu(A_{jk}) =
\prod_{j=1}^{+\infty} \left( \prod_{k=1}^{+\infty} \mu(A_{jk}) \right)= \prod_{j=1}^{+\infty} \mu(E_j).
\]
This completes the proof.
\end{proof}

We also have the Carath\'eodory's characterization for $\mu$-measurable sets of $\mathbb{R}_{>0}$. The proof of the following theorem is similar to that of Theorem 3.30, in \cite{zygmund}, but reformulated to our setting.

\begin{theorem} A set $E$ of $\mathbb{R}_{>0}$ is $\mu$-measurable if and only if for every set $A$ of $\mathbb{R}_{>0}$
\[
\mu_e(A) = \mu_e(A \cap E) \cdot \mu_e(A \setminus E).
\]
\end{theorem}

We summarize our main result in the following theorem.

\begin{theorem} Let $\mathbb{R}_{>0}$ with the usual topology. Then there exist a $\sigma$-algebra $\mathcal{M}$ containing every Borel set 
of $\mathbb{R}_{>0}$ and a measure $\mu : \mathcal{M} \to [1, +\infty]$ such that
\begin{enumerate}
\item[(i)] $\mu(\emptyset) = 1$.

\item[(ii)] For each interval $I$ of $\mathbb{R}_{>0}$, we have that $\mu(I) = \ell(I)$ if $I$ is bounded and $\inf I >0$; opposite case $\mu(I)=+\infty$.

\item[(iii)] 
\[
\mu \left( \bigcup_j E_j \right) = \prod_j \mu(E_j)
\]
for every countable collection $\{ E_j \}$ of pairwise disjoint sets of $\mathcal{M}$. We call to this property: 
countable multiplicativity. 
\end{enumerate}
\end{theorem}

By this Theorem, we say that $\mu$ is a \textit{multiplicative measure}. So, we have constructed a multiplicative measure on 
$\mathbb{R}_{>0}$ using only the topology usual on $\mathbb{R}_{>0}$ and its multiplicative structure. 

%Maybe this construction is anecdotal, but it can be done.

\

In the next section, we shall describe the connection between our measure $\mu$ and the Lebesgue measure. 

\section{The measure $\lambda$}

Let $\lambda$ be now the (multiplicative) measure defined on the $\sigma$-algebra of Lebesgue measurable subsets of $\mathbb{R}_{>0}$ by

\[
\lambda(E) = \exp\left(\int_{E} \frac{1}{x} \, dx \right).
\]
\\
It is clear that for each interval $I=[a,b] \subset \mathbb{R}_{>0}$, $\mu(I) = \lambda(I)$. So, $\mu(G) = \lambda(G)$ for each open set 
$G \subset \mathbb{R}_{>0}$. If we show that the $\sigma$-algebra $\mathcal{M}$ coincides with the $\sigma$-algebra of the Lebesgue measurable sets of $\mathbb{R}_{>0}$, then it will follow that $\mu(E) = \lambda(E)$ for all $E \in \mathcal{M}$.

From the definition of the $\mu$-measurability follows that $E \in \mathcal{M}$ (or $E$ is \textit{$\mu$-measurable}) if and only if $E = H \setminus U$ where $H$ is of type $G_{\delta}$ and $\mu(U) = 1$. Thus, to prove that $E$ is $\mu$-measurable if and only if $E$ is Lebesgue measurable it suffices to show that

\[
\mu(U) =1  \,\,\,\, \text{if and only if} \,\,\,\,\, \left| U \right|_e = 0,
\]
\\
here $| \cdot |_e$ denote the outer Lebesgue measure. 
Suppose $\mu(U) = 1$. By Theorem \ref{G delta}, claim (i), given $\epsilon > 0$, there exists an open set $G$, such that $U \subset G$ and

\begin{equation}\label{la G}
\mu(G) \leq 1 + \epsilon.
\end{equation}
\\
Theorem 1.1, in \cite{zygmund}, implies there exists a collection of nonoverlapping closed interval $\{ I_j = [a_j, b_j] \}$ such that 
$G = \bigcup_j I_j$. Then, $\log(G) \subset \bigcup_j [\log(a_j), \log(b_j)]$. Since $\lambda(G)=\mu(G)$, we obtain,
by (\ref{la G}), that

\[
| \log(G) |_{e} \leq \sum_{j} [\log(b_j) - \log(a_j) ] = \log\left(\lambda(G)\right) = \log (\mu(G)) \leq 
\log(1+\epsilon).
\]
So, $| \log(G) |_{e} = 0$ and, hence, $| \log(U) |_{e} = 0$. Then, to apply the exponential map on the region $\log(U)$, we have 
$|U|_e = 0$.

Similarly, it is proved that $|U|_e=0$ implies $\mu(U) =1$. Thus, the $\sigma$-algebra $\mathcal{M}$ coincides with the $\sigma$-algebra of the Lebesgue measurable sets of $\mathbb{R}_{>0}$, and $\mu \equiv \lambda$.

\

\end{document}